\documentclass{amsart}

\usepackage{amsmath,amssymb,mathrsfs}

\makeatletter
\def\@settitle{\begin{center}\baselineskip14\p@\relax\bfseries{\large\@title}\thispagestyle{empty}\end{center}}
\def\@setauthors{%
  \begingroup
  \def\thanks{\protect\thanks@warning}%
  \trivlist
  \centering\footnotesize \@topsep30\p@\relax
  \advance\@topsep by -\baselineskip
  \item\relax
  \author@andify\authors
  \def\\{\protect\linebreak}%
  {\authors}%
  \ifx\@empty\contribs
  \else
    ,\penalty-3 \space \@setcontribs
    \@closetoccontribs
  \fi
  \endtrivlist
  \endgroup
}
\def\maketitle{\par
  \@topnum\z@ 
  \@setcopyright
  \thispagestyle{firstpage}
  \ifx\@empty\shortauthors \let\shortauthors\shorttitle
  \else \andify\shortauthors
  \fi
  \@maketitle@hook
  \begingroup
  \@maketitle
  \toks@\@xp{\shortauthors}\@temptokena\@xp{\shorttitle}%
  \toks4{\def\\{ \ignorespaces}}
  \edef\@tempa{%
    \@nx\markboth{\the\toks4
      \@nx
      {\the\toks@}}{\the\@temptokena}}%
  \@tempa
  \endgroup
  \c@footnote\z@
  \@cleartopmattertags
}
\makeatother

\newtheorem{proposition}{proposition}[section]
\newtheorem{theorem}[proposition]{Theorem}
\newtheorem{lemma}[proposition]{Lemma}
\newtheorem{corollary}[proposition]{Corollary}
\theoremstyle{definition}

\newtheorem{example}[proposition]{Example}
\newtheorem{remark}[proposition]{Remark}

\newcommand{\Q}{\mathbb Q}
\newcommand{\R}{\mathbb R}
\newcommand{\C}{\mathbb C}
\newcommand{\ce}{c_{\ue}}
\newcommand{\Digit}{\mathcal{A}}

\newcommand{\Fourier}[1][f]{\mathcal{F}[#1]}
\newcommand{\rapid}[1][V]{\mathscr{S}(#1)}
\newcommand{\lzf}[3][\ue]{\Phi_{#1}(#2;\,#3)}
\newcommand{\dlzf}[3][\ud]{\Phi^*_{#1}(#2;\,#3)}
\newcommand{\vlzf}[2]{\bs{\Phi}(#1;\,#2)}
\newcommand{\vdlzf}[2]{\bs{\Phi}^*(#1;\,#2)}
\newcommand{\clzf}[2]{\bs{\Psi}(#1;\,#2)}
\newcommand{\cdlzf}[2]{\bs{\Psi}^*(#1;\,#2)}
\newcommand{\sign}[1][\ul{a}]{(\sqrt{-1})^{|#1|}}

\newcommand{\pmat}[1]{\begin{pmatrix} #1 \end{pmatrix}}

\newcommand{\set}[2]{\left\{#1 ;\; #2 \right\}}
\newcommand{\transpose}[1]{\mskip2.4 mu {}^t \mskip -2.5 mu #1}

\newcommand{\s}{\ul{s}}

\newcommand{\efactor}{\mathcal{E}}
\newcommand{\w}{\ul{w}}

\newcommand{\kvsup}[1]{\kappa_{\ue}(#1)}
\newcommand{\localZ}[2][\ul{b}]{Z_{#1}(#2;\,\ul{s})}
\newcommand{\dlocalZ}[2][\ul{c}]{Z^*_{#1}(#2;\,\ul{s})}

\newcommand{\ul}[1]{\underline{#1}}
\newcommand{\ud}{\boldsymbol{\delta}}
\newcommand{\ue}{\boldsymbol{\varepsilon}}

\newcommand{\bs}[1]{\boldsymbol{#1}}

\newcommand{\Ir}[1][r]{\mathcal{I}_{#1}}

\newcommand{\Oe}[1][\ue]{\mathcal{O}_{#1}}

\newcommand{\kv}[1][\ul{a}]{\bs{\kappa}(#1)}
\newcommand{\kvsub}[2][\ul{a}]{\kappa_{#2}(#1)}
\newcommand{\pcos}[2][\ul{a}]{Q_{#1}(#2)}
\newcommand{\Pcos}[2][\ul{a}]{Q_{#1}{#2}}

\newcommand{\uluelement}[1][j]{w_j}
\newcommand{\innV}[2]
	{\left\langle \mskip2.0 mu #1 \mskip2.0 mu | 
	\mskip2.0 mu #2 \mskip2.0 mu \right\rangle}

\title
{
Completion of local zeta functions associated with a certain class of homogeneous cones}

\author[H.\ Nakashima]
{Hideto Nakashima}
\address{Graduate School of Mathematics,  
Nagoya University\\ 
Furo-cho, Chikusa-ku, Nagoya 464-8602, 
Japan
}
\email{h-nakashima@math.nagoya-u.ac.jp}
\keywords{
Zeta functions,
functional equations,
prehomogeneous vector spaces,
homogeneous cones.
}

\subjclass[2010]{
11M41, 
11S90, 
22E25. 
}

\begin{document}

\begin{abstract}
It is well known that the Riemann zeta function 
can be completed to the Riemann xi function $\xi(s)$
in the sense that its functional equation has a higher symmetric form
$\xi(1-s)=\xi(s)$.
In the previous paper (to appear in Tohoku Math.\ J.),
we give an explicit formula of functional equations between local and global zeta functions associated with a homogeneous cone and with its dual cone.
In this paper, we consider a completion of these local zeta functions and
show that, for a certain class of homogeneous cones,
the associated local zeta functions admit a kind of completion forms.
\end{abstract}

\maketitle

\section{Introduction}

A theory of prehomogeneous vector spaces, constructed by Sato~\cite{MSato} (see also Sato--Shintani~\cite{MSatoShintani}, Kimura~\cite[Introduction]{Kimura}),
provides a systematic method to construct zeta functions satisfying functional equations.
The Riemann zeta function $\zeta(s)$ can be viewed as a typical example of such zeta functions, 
and it has many interesting and important properties.
Among them, we focus on the functional equation and its completion, that is,
$\zeta(s)$ can be completed to the Riemann xi function 
$\xi(s):=
2^{-1}s(s-1)
\pi^{-s/2}\Gamma(s/2)\zeta(s)$
whose functional equation has a higher symmetric form
$\xi(1-s)=\xi(s)$.
Here, $\Gamma(s)$ is the ordinary gamma function.
Therefore, as analogous to $\zeta(s)$,
it is natural to ask whether or not
zeta functions associated with prehomogeneous vector spaces can be completed.
Let us state this problem more precisely.
Let $\bs{\zeta}(\s)$ and $\bs{\zeta}^*(\s)$ $(\s\in\C^r)$ be
vector-valued zeta functions associated with a regular prehomogeneous vector space
and with its dual prehomogeneous vector space, respectively.
Then, a functional equation between these zeta functions is written as
\[\bs{\zeta}^*(\tau(\s))=A(\s)\bs{\zeta}(\s),\]
where 
$\tau$ is a suitable affine transformation on $\C^r$ 
and $A(\s)$ is a suitable coefficient matrix.
For
these zeta functions $\bs{\zeta}(\s)$ and $\bs{\zeta}^*(\s)$,
let us find square matrices $B(\s)$ and $B^*(\s)$ such that
\[\bs{\xi}^*(\tau(\s))=\efactor\bs{\xi}(\s)\quad\text{where}\quad
\bs{\xi}(\s):=B(\s)\bs{\zeta}(\s)\text{ and }\bs{\xi}^*(\s):=B^*(\s)\bs{\zeta}^*(\s),\]
where $\efactor$ is the so-called $\varepsilon$-factor,
that is, a diagonal matrix with entries $\exp(\theta\sqrt{-1})$ $(\theta\in\R)$.
Since $B(\ul{s})$ and $B^*(\ul{s})$ may be different matrices,
there is a trivial solution $B(\s)=A(\s)$, $B^*(\s)=I$ and $\efactor=I$ ($I$ is the identity matrix)
and hence we would like to make $B(\s)$ and $B^*(\s)$ as similar as possible.
For such cases,
we say that a pair $(\bs{\zeta}(\s),\bs{\zeta}^*(\s))$ can be completed to $(\bs{\xi}(\s),\bs{\xi}^*(\s))$,
or the pair $(\bs{\xi}(\s),\bs{\xi}^*(\s))$ is a completion of $(\bs{\zeta}(\s),\bs{\zeta}^*(\s))$ in this paper.
We shall explain this notion by a concrete example.
Let $(G,\rho,V)$ be a prehomogeneous vector space in Sato~\cite[\S\S7.1 Example (A)]{FSatoI}.
We use all notation in that paper without comments, and assume $v(L^{(1)*})=1$ for simplicity.
For dual zeta functions, we choose $E$ as a $\Q$-regular subspace.
Put $\tau(\ul{s}):=(s_1+s_2+s_3-1,\ 1-s_3,\ 1-s_2)$ for $\ul{s}=(s_1,s_2,s_3)\in\C^3$.
Then, we have 
\[
\pmat{\xi_+(L^*_E;\,\tau(\ul{s}))\\ \xi_-(L^*_E;\,\tau(\ul{s}))}
=
\frac{2\Gamma(s_2)\Gamma(s_3)}{(2\pi)^{s_2+s_3}}
\pmat{
\cos\bigl(\frac{\pi}{2}(s_2+s_3)\bigr)
&
\sin\bigl(\frac{\pi}{2}(s_2-s_3)\bigr)\\
\sin\bigl(\frac{\pi}{2}(s_2-s_3)\bigr)
&
\cos\bigl(\frac{\pi}{2}(s_2+s_3)\bigr)}
\pmat{\xi_+(L;\ul{s})\\ \xi_-(L;\,\ul{s})}
\]
by Theorem~3 (iii) of that paper~\cite{FSatoI}.
This coefficient matrix can be obviously diagonalized, and moreover using a formula \eqref{eq:MF} of the gamma function, we obtain a completed functional equation in such a way that
by setting
\[
B(\ul{s})=B^*(\ul{s})=\pi^{-\tfrac{s_2+s_3}{2}}
\pmat{\Gamma\bigl(\frac{s_2}{2}\bigr)\Gamma\bigl(\frac{s_3}{2}\bigr)&0\\
0&\Gamma\bigl(\frac{s_2+1}{2}\bigr)\Gamma\bigl(\frac{s_3+1}{2}\bigr)}
\cdot\pmat{1&1\\1&-1},
\]
we have
\[
\bs{\eta}^*_E\bigl(\tau(\ul{s})\bigr)
=
\pmat{1&0\\0&-1}
\bs{\eta}_E(\ul{s}),\]
where
\[\bs{\eta}^*_E(\ul{s})=B^*(\ul{s})\pmat{\xi_+(L^*_E;\,\ul{s})\\ \xi_-(L^*_E;\,\ul{s})},\quad
\bs{\eta}_E(\ul{s})
=
B(\ul{s})\pmat{\xi_+(L;\ul{s})\\ \xi_-(L;\,\ul{s})}.
\]
This observation seems to be new.
Other dual zeta functions in \S\S7.1 of \cite{FSatoI} can be also completed, but for another example \cite[\S\S7.2 Example (B)]{FSatoI}, not all functional equations seem to be completed.
In general,
it is not known whether prehomogeneous zeta functions can be completed or not, 
except for some particular reductive cases
(cf.\ \cite{SatakeFaraut}, \cite{DatsovskyWright} and \cite{Thor}).

In this paper,
we consider this problem for prehomogeneous vector spaces
associated with homogeneous open convex cones containing no entire line
(homogeneous cones for short in what follows)
studied in the previous paper~\cite{N2018},
and we show that,
for a certain class of homogeneous cones,
the associated zeta functions have completions.
Since functional equations of prehomogeneous zeta functions essentially come from those of local zeta functions,
we actually work with local zeta functions.
It is worthy to mention that our prehomogeneous vector spaces
are not reductive but solvable.
We are therefore required additional observations which are not needed in reductive cases;
for example,
we will make a delicate discussion about orders of zeta distributions (see~\eqref{def:zd} for definition) when we consider functional equations with respect to them (see Section~\ref{sect:zeta distribution}).

We now introduce the terminologies and the notations which we need to state the precise argument.
Let $\Omega$ be a homogeneous cone of rank $r$ in an $n$-dimensional real vector space $V$.
By Vinberg~\cite{Vinberg},
there exists a split solvable Lie group $H$ acting on $\Omega$ linearly and simply transitively.
According to Ishi~\cite{Ishi2006}, 
we realize $\Omega$ as a subset of the open convex cone $\mathcal{S}^+_N$ of positive-definite symmetric matrices of size $N\in\mathbb{N}$,
that is,
taking $r$ suitable integers $n_1,\dots,n_r$
and 
a suitable system    
$\mathcal{V}_{kj}\subset \mathrm{Mat}(n_k,n_j;\ \R)$
$(1\le j<k\le r)$
of vector spaces,
we regard $\Omega$ as a homogeneous cone included in a subspace of symmetric matrices as follows:
\begin{equation}
\label{eqV}
\Omega\cong
\set{
x=\pmat{x_1I_{n_1}&\transpose{X}_{21}&\cdots&\transpose{X}_{r1}\\
X_{21}&x_2I_{n_2}&\ddots&\transpose{X}_{r2}\\
\vdots&\ddots&\ddots&\vdots\\
X_{r1}&X_{r2}&\cdots&x_rI_{n_r}}}{
\begin{array}{l}
x_j\in\R\\\quad(j=1,\dots,r)\\
X_{kj}\in\mathcal{V}_{kj}\\
\quad(1\le j<k\le r)
\end{array}}\cap\mathcal{S}^+_N,
\end{equation}
where $N:=n_1+\cdots+n_r$.
The following integers and vectors are used frequently in this paper without any comments:
\[
\begin{array}{c}
\displaystyle
n_{kj}:=\dim \mathcal{V}_{kj},\quad
p_k:=\sum_{j<k}n_{kj},\quad
q_j:=\sum_{k>j}n_{kj},\\[1em]
\displaystyle
\ul{1}=(1,\dots,1),\quad
\ul{p}=(p_1,\dots,p_r),\quad
\ul{q}=(q_1,\dots,q_r),\quad
\ul{d}=\ul{1}+\frac{1}{2}(\ul{p}+\ul{q}).
\end{array}
\]

Put $\Ir:=\{\pm1\}^r$.
For each $\ue=(\varepsilon_1,\dots,\varepsilon_r)\in\Ir$,
we denote by $\Oe$ an orbit of $H$ through $\mathrm{diag}(\varepsilon_1I_{n_1},\dots,\varepsilon_rI_{n_r})$.
Note that $\Omega=\Oe[(1,\dots,1)]$.
Then, $\bigsqcup_{\ue\in\Ir}\Oe$ is a Zariski open set in $V$ (cf.\ Gindikin~\cite[p.\ 77]{Gindikin64}) so that
$V$ is a real prehomogeneous vector space.
Let $\Delta_1(x),\dots,\Delta_r(x)$ be the basic relative invariants of $\Omega$.
We denote by $\rapid$ the Schwartz space of rapidly decreasing functions on $V$.
For $f\in\rapid$, we put
\begin{equation}
\label{eq:lzf}
\lzf{f}{\ul{s}}:=\int_{\Oe}|\Delta_1(x)|^{s_1}\cdots|\Delta_r(x)|^{s_r}f(x)\,d\mu(x)\quad(\ue\in\Ir)
\end{equation}
which are called the local zeta functions
associated with $\Oe$.
Here, $d\mu(x)$ is a suitable invariant measure on $\Oe$.
It is known that $\lzf{f}{\ul{s}}$
are absolutely convergent for $\mathrm{Re}\,\ul{s}>\ul{d}\sigma^{-1}$,
and analytically continued to meromorphic functions of $\ul{s}$ in the whole space $\C^r$ (cf.\ \cite{BG69}).
Here, 
$\sigma=(\sigma_{jk})_{1\le j,k\le r}$ is a unimodular matrix containing information about the basic relative invariants as
\begin{equation}
\label{eq:multiplier matrix}
\Delta_j\bigl(\mathrm{diag}(x_1I_{n_1},\dots,x_rI_{n_r})\bigr)
=
x_{1}^{\sigma_{j1}}\cdots x_r^{\sigma_{jr}}\quad(x_1,\dots,x_r\in\R^\times),
\end{equation}
which is called the multiplier matrix of $\Omega$
(cf.\ \cite{N2014}).
Note that we write $\ul{\alpha}>\ul{\beta}$ for $\ul{\alpha},\,\ul{\beta}\in\R^r$ if $\alpha_j>\beta_j$ for all $j=1,\dots,r$.
Associated with the dual prehomogeneous vector space of $V$,
we also have local zeta functions $\dlzf{f}{\ul{s}}$ for $\ud\in\Ir$ and $f\in\rapid$, which are also analytically continued to meromorphic functions of $\ul{s}\in\C^r$.
We shall write these local zeta functions in a vector form as
\[\vlzf{f}{\ul{s}}=\bigl(\lzf{f}{\ul{s}}\bigr)_{\ue\in\Ir}\quad
\text{and}\quad
\vdlzf{f}{\ul{s}}=\bigl(\dlzf{f}{\ul{s}}\bigr)_{\ud\in\Ir},\]
where we equip $\Ir$ with a total order, and fix it.

Let us take and fix a suitable inner product $\innV{\cdot}{\cdot}$ in $V$.
The dual vector space $V^*$ of $V$ is identified with $V$ through this inner product.
The Fourier transform $\Fourier$ of $f\in\rapid$ is defined as
\[
\Fourier(x):=\int_Vf(y)\exp(2\pi\sqrt{-1}\innV{x}{y})\,dy,
\]
where $dy$ is the Euclidean measure on $V$.
Let $\tau$ be the affine transformation on $\C^r$ defined by
$\tau(\s):=(\ul{d}-\s\sigma)\sigma_*^{-1}$,
where $\sigma_*$ is the multiplier matrix of the dual cone $\Omega^*$ of $\Omega$.
For $\ul{\alpha}\in\C^r$,
we write 
$\Gamma(\ul{\alpha}):=\Gamma(\alpha_1)\cdots\Gamma(\alpha_r)$.
Then, the Gindikin gamma function $\Gamma_\Omega(\ul{\alpha})$
of $\Omega$ is defined as
\begin{equation}
\label{eq:defofgamma}
\Gamma_{\Omega}(\ul{\alpha})
=
(2\pi)^{(n-r)/2}\Gamma\Bigl(\ul{\alpha}-\frac{1}{2}\ul{p}\Bigr)
\end{equation}
(cf.\ Gindikin~\cite{Gindikin64}).
Moreover,
we set
\[
A(\ul{\alpha})=\biggl(
\exp\Bigl\{\frac{\pi\sqrt{-1}}{2}\Bigl(
\sum_{j=1}^r\varepsilon_j\delta_j\alpha_j
+\frac{1}{2}\sum_{1\le j<k\le r}\varepsilon_j\delta_kn_{kj}
\Bigr)\Bigr\}
\biggr)_{\ud,\ue\in\Ir},
\]
where the index $\ue$ runs horizontally and $\ud$ vertically.
Then,
Proposition 4.3 of the previous paper \cite{N2018} gives
the following functional equation of local zeta functions:
\begin{equation}
\label{eq:FE}
\vlzf{\Fourier}{\ul{s}}
=
\frac{\Gamma_\Omega(\ul{s}\sigma)}{(2\pi)^{|\ul{s}\sigma|}}
A\Bigl(\ul{s}\sigma-\frac{1}{2}\ul{p}\Bigr)\vdlzf{f}{\tau(\ul{s})}.
\end{equation}
Here,
we set $|\ul{\alpha}|:=\alpha_1+\cdots+\alpha_r$ for $\ul{\alpha}\in\C^r$.

For a technical reason, we assume the following condition:
\begin{equation}
\label{assumption}
\text{For fixed $m=0, 1$,
one has $\frac{\pi}{4}\sum_{j<k}\varepsilon_j\delta_kn_{kj}\equiv m\pi$ $({\rm mod}\ 2\pi)$ for all $\ue,\ud\in\Ir$.}
\end{equation}
Set $\Digit:=\{0,1\}^r$.
We equip $\Digit$ with a total order, and fix it.
According to this order, we define a diagonal matrix 
$\Lambda(\ul{\alpha})$ 
for $\ul{\alpha}\in\C^r$ by
\begin{equation}
\label{eq:defofD}
\Lambda(\ul{\alpha}):=\mathrm{diag}\left(\frac{\pi^{|\ul{\alpha}|/2}}{\Gamma\bigl(2^{-1}(\ul{\alpha}+\ul{a})\bigr)}\right)_{\ul{a}\in\Digit}.
\end{equation}
Then, the main theorem is stated as follows.

\begin{theorem}
\label{theoremB}
Assume the condition~\eqref{assumption}.
Then, there exists an orthogonal matrix $J$ such that,
by setting
\[
\clzf{f}{\s}:=\Lambda\bigl(\s\sigma-2^{-1}\ul{p}\bigr){}^{\,t\!}J\vlzf{f}{\s},\quad
\cdlzf{f}{\s}:=
\Lambda\bigl(\s\sigma_*-2^{-1}\ul{q}\bigr){}^{\,t\!}J\vdlzf{f}{\s},
\]
one has
\[
\clzf{\Fourier}{\s}=\efactor\,\cdlzf{f}{\tau(\s)}\quad(\s\in\C^r,\ f\in\rapid),
\]
where
\begin{equation}
\label{def:efactor}
\efactor=(-1)^m\mathrm{diag}\Bigl(\sign\Bigr)_{\ul{a}\in\Digit}.
\end{equation}
\end{theorem}

The transformation matrices $B(\ul{s})$ and $B^*(\ul{s})$ in this theorem are not same but similar in a sense that
the difference, that is the arguments of $\Lambda$ relate to the domains in which the corresponding gamma functions $\Gamma_{\Omega}$ and $\Gamma_{\Omega^*}$ of $\Omega$ and $\Omega^*$, respectively, converge absolutely.
Namely, we see by \cite[p.\ 22]{Gindikin64} that
the integral
\[
\int_{\Omega}\Delta_1(x)^{s_1}\cdots\Delta_r(x)^{s_r}e^{-\innV{x}{I_N}}d\mu(x)
=
\Gamma_{\Omega}(\ul{s}\sigma)\quad(I_N:=\mathrm{diag}(I_{n_1},\dots,I_{n_r}))
\]
converges absolutely when $\mathrm{Re}\,\ul{s}\sigma-\frac{1}{2}\ul{p}>0$, 
whereas the integral
\[
\int_{\Omega^*}\Delta^*_1(y)^{s_1}\cdots\Delta^*_r(y)^{s_r}e^{-\innV{y}{I_N}}d\mu^*(y)=\Gamma_{\Omega^*}(\ul{s}\sigma_*)
\]
converges absolutely when $\mathrm{Re}\,\ul{s}\sigma_*-\frac{1}{2}\ul{q}>0$. 
Here, $\Delta^*_1(y),\dots,\Delta^*_r(y)$ are the basic relative invariants of $\Omega^*$ and $d\mu^*(y)$ a suitable invariant measure on $\Omega^*$. 

We shall prove Theorem~\ref{theoremB} in Section~\ref{sect:proof}.
Section~\ref{sect:zeta distribution} is devoted to investigating a relationship between ${}^{t\!}J\vlzf{f}{\s}$ and zeta distributions (see \eqref{def:zd} for definition).

\section{Proof of Theorem~\ref{theoremB}}
\label{sect:proof}

Let us start proving Theorem~\ref{theoremB}.
We use all notations in Introduction.
Set $\w=\ul{s}\sigma-\frac{1}{2}\ul{p}$.
Then, the gamma matrix, that is
the coefficient matrix in \eqref{eq:FE} can be written, by~\eqref{eq:defofgamma}, as
\[
\frac{\Gamma_\Omega(\ul{s}\sigma)}{(2\pi)^{|\ul{s}\sigma|}}
A\Bigl(\ul{s}\sigma-\frac{1}{2}\ul{p}\Bigr)
=
\frac{\Gamma(\ul{s}\sigma-\frac{1}{2}\ul{p})}{(2\pi)^{|\ul{s}\sigma-\frac{1}{2}\ul{p}|}}A\Bigl(\ul{s}\sigma-\frac{1}{2}\ul{p}\Bigr)
=
\frac{\Gamma(\w)}{(2\pi)^{|\w|}}A(\w).
\]
In the first equality,
we use $|\ul{p}|=n-r$.
Throughout this paper,
we always assume the condition~\eqref{assumption}
so that
$A(\ul{\alpha})$ reduces to
\begin{equation}
\label{def:gamma matrix}
A(\ul{\alpha})=
(-1)^m\biggl(\exp\Bigl(\frac{\pi\sqrt{-1}}{2}\sum_{j=1}^r\varepsilon_j\delta_j\alpha_j\Bigr)\biggr)_{\ud,\ue\in\Ir}.
\end{equation}

For $\ul{a}\in\Digit$,
let $\kv\in\Ir[2^r]$ be a column vector defined by
\[
\kv:=\bigl(\kvsub{\ue}\bigr)_{\ue\in\Ir}\in\Ir[2^r],\quad\text{where}\quad
\kvsub{\ue}:=\prod_{j=1}^r\varepsilon_j^{a_j}\quad(\ue\in\Ir).
\]
Notice $\set{\kv}{\ul{a}\in\Digit}=\mathcal{I}_2^{\otimes r}$ so that
it forms an orthogonal basis of $\R^{2^r}$ with respect to the standard inner product in $\R^{2^r}$.
Let $J$ be an orthogonal matrix of size $2^r$ obtained by arraying column vectors $2^{-r/2}\kv$ in a row, that is,
\[
J:=2^{-r/2}\bigl(\kv\bigr)_{\ul{a}\in\Digit}.
\]
For $\ul{a}\in\Digit$, we set
\[\pcos{\ul{\alpha}}:=\prod_{j=1}^r\cos\Bigl(\frac{\pi}{2}(\alpha_j-a_j)\Bigr)\quad(\ul{\alpha}\in\C^r).\]

\begin{lemma}
\label{lemma:diagonalization}
The matrix $A(\ul{\alpha})$ in~\eqref{def:gamma matrix} can be diagonalized by $J$ independent of $\ul{\alpha}$ as follows:
\[
{}^{t\!}JA(\ul{\alpha})J=(-1)^m\cdot 2^r\mathrm{diag}\Bigl(\sign\Pcos(\ul{\alpha})\Bigr)_{\ul{a}\in\Digit}.
\]
\end{lemma}
\begin{proof}
We shall prove this lemma by showing that,
for any $\ul{a}\in\Digit$, 
the vector $\kv$ is an eigenvector of $A(\ul{\alpha})$ and its corresponding eigenvalue is $(-1)^m\cdot 2^r\sign \pcos{\ul{\alpha}}$.

For any vectors $\ul{a}\in\Digit$ and $\ul{\beta}=(\beta_1,\dots,\beta_r)\in\C^r$,
it is easily verified that
\[
\sum_{\ue\in\Ir}\kvsub{\ue}\exp\Bigl(\frac{\pi\sqrt{-1}}{2}\sum_{j=1}^r\varepsilon_j\beta_j\Bigr)
=
2^r(\sqrt{-1})^{|\ul{a}|}\prod_{j=1}^r\Pcos(\ul{\beta}).
\]
In our case,
the vector $\ul{\beta}$ is one of the vectors $(\delta_1\alpha_1,\dots,\delta_r\alpha_r)$ for some $\ud\in\Ir$.
An elementary calculation yields that
\[
\cos\Bigl(\frac{\pi}{2}(\delta z-a)\Bigr)
=
\delta^a
\cos\Bigl(\frac{\pi}{2}(z-a)\Bigr)\quad\text{for $\delta\in\{1,-1\}$ and $a\in\{0,1\}$},
\]
so that we obtain
$\Pcos(\ul{\beta})=\kvsub{\ud}\Pcos(\ul{\alpha})$.
This implies that,
for any $\ul{a}\in\Digit$, 
the vector $\kv$ is an eigenvector of $A(\ul{\alpha})$,
and $(-1)^m\cdot 2^r\sign \pcos{\ul{\alpha}}$ is its corresponding eigenvalue,
and hence the lemma is proved.
\end{proof}

Recall that $\w=\s\sigma-\frac{1}{2}\ul{p}$.

\begin{lemma}
\label{lemma:halfGamma}
For each $\ul{a}\in\Digit$, one has
\[
\frac{\Gamma(\w)}{(2\pi)^{|\w|}}\pcos{\w}=
\frac{\pi^{r/2-|\w|}}{2^r}\cdot\frac{\Gamma\bigl(2^{-1}(\ul{s}\sigma-2^{-1}\ul{p}+\ul{a})\bigr)}{\Gamma\bigl(2^{-1}(\tau(\ul{s})\sigma_*-2^{-1}\ul{q}+\ul{a})\bigr)}.
\]
\end{lemma}
\begin{proof}
Let us recall two famous formulas of the gamma function,
that is, Euler's reflection formula and
Legendre's duplication formula:
\[
\Gamma(z)\Gamma(1-z)=\frac{\pi}{\sin\pi z},\quad
\Gamma(z)=\frac{2^z}{2\sqrt{\pi}\ }\Gamma\Bigl(\frac{z}{2}\Bigr)
\Gamma\Bigl(\frac{z+1}{2}\Bigr)\quad(z\in \C).
\]
Combining these two formulas, we obtain the following formula
\begin{equation}
\label{eq:MF}
\Gamma(z)\,\cos\Bigl(\frac{\pi}{2}(z-a)\Bigr)=
\frac{2^z\sqrt{\pi}}{2}\cdot
\frac{\Gamma\bigl(2^{-1}(z+a)\bigr)}{\Gamma\bigl(2^{-1}(1-z+a)\bigr)}
\quad(z\in\C;\ a=0\text{ or }1).
\end{equation}
This equation implies that
\[
\prod_{i=1}^r
\Gamma(w_i)\cos\left(\frac{\pi}{2}(w_i-a_i)\right)
=
\prod_{i=1}^r
\frac{2^{w_i}\sqrt{\pi}}{2}\cdot
\frac{\Gamma\bigl(2^{-1}(w_i+a_i)\bigr)}{\Gamma\bigl(2^{-1}(1-w_i+a_i)\bigr)},
\]
that is,
\[
\Gamma(\w)\pcos{\w}
=
\frac{2^{|\w|}\pi^{r/2}}{2^r}\cdot
\frac{\Gamma\bigl(2^{-1}(\w+\ul{a})\bigr)}{\Gamma\bigl(2^{-1}(\ul{1}-\w+\ul{a})\bigr)}
\quad(\ul{a}\in\mathcal{A}).
\]
The facts $\ul{d}=\ul{1}+(\ul{p}+\ul{q})/2$ and $\ul{d}-\ul{s}\sigma=\tau(\ul{s})\sigma_*$ show
\[
\ul{1}-\ul{w}
=
\ul{1}-\ul{s}\sigma+\frac{1}{2}\ul{p}
=
\ul{d}-\ul{s}\sigma-\frac{1}{2}\ul{q}
=
\tau(\ul{s})\sigma_*-\frac{1}{2}\ul{q},
\]
whence we arrive at the formula in the lemma.
\end{proof}

We are now in the final step of the proof of the main theorem.
Let us recall definition of diagonal matrices $\Lambda(\ul{\alpha})$
and $\efactor$ in \eqref{eq:defofD} and \eqref{def:efactor}, respectively.
We set $\ul{v}=\ul{1}-\ul{w}=\tau(\ul{s})\sigma_*-\frac{1}{2}\ul{q}$ for brevity.
Notice that $\frac{r}{2}-|\w|=-\frac{1}{2}\bigl|\ul{w}\bigr|
+\frac{1}{2}\bigl|\ul{v}\bigr|$.
Then,
Lemmas~\ref{lemma:diagonalization} and~\ref{lemma:halfGamma} 
yield that
\[
\begin{array}{r@{\ }c@{\ }l}
\displaystyle
\frac{\Gamma(\w)}{(2\pi)^{|\w|}}A(\w)
&=&
\displaystyle
(-1)^m\cdot
2^r
J
\mathrm{diag}\left(
	\sign
	\frac{\Gamma(\w)}{(2\pi)^{|\w|}}
	\Pcos(\w)
\right)_{\ul{a}\in\Digit}
{}^{\,t\!}J\\[1em]
&=&
\displaystyle
(-1)^m\cdot
2^r
J
\mathrm{diag}\left(
	\sign
	\frac{\pi^{r/2-|\w|}}{2^r}\cdot
	\frac{\Gamma\bigl(2^{-1}(\ul{w}+\ul{a})\bigr)}%
	{\Gamma\bigl(2^{-1}(\ul{v}+\ul{a})\bigr)}
\right)_{\ul{a}\in\Digit}
{}^{\,t\!}J\\[1.1em]
&=&
\displaystyle
J
\mathrm{diag}\left(
	(-1)^m\,\sign \,
	\frac{\Gamma(2^{-1}\bigl(\ul{w}+\ul{a})\bigr)}{\pi^{|\ul{w}|/2}}
	\,
	\frac{\pi^{|\ul{v}|/2}}{\Gamma\bigl(2^{-1}(\ul{v}+\ul{a})\bigr)}
\right)_{\ul{a}\in\Digit}
{}^{\,t\!}J\\[1em]
&=&
\displaystyle
J
\Lambda(\ul{w})^{-1}
\efactor
\Lambda(\ul{v})
{}^{\,t\!}J,
\end{array}
\]
whence
we conclude that the functional equation~\eqref{eq:FE} can be completed.
\qed

\begin{remark}
We note that there exist homogeneous cones satisfying~\eqref{assumption}.
In fact, 
the exceptional symmetric cone $\mathrm{Herm}(3,\mathbb{O})^+$ obviously satisfies~\eqref{assumption}.
Other examples are given as homogeneous cones $\Omega$ such that
$n_{kj}=0$ or $4$ for all $j<k$.
Such homogeneous cones can be constructed from chordal and $A_4$-free graphs 
(see Letac--Massam~\cite{LM2007} for definition) as follows.
Let $G$ be a chordal and $A_4$-free graph of size $n$.
Set 
\[V_G:=\set{x\in\mathrm{Herm}(n,\mathbb{H})}{x_{ij}=0\text{ if }i\not\sim j\text{ in }G}.
\]
Then, $\Omega_G:=V_G\cap\mathrm{Herm}(n,\mathbb{H})^+$ is a desired homogeneous cone.
\end{remark}

\begin{remark}
Under the assumption~\eqref{assumption},
we are able to give an explicit formula to the determinant of gamma matrices $A(\ul{\alpha})$.
In fact, let us assume that $r\ge2$ and set $\Digit'=\set{\ue\in\Digit}{\varepsilon_1=1}$.
Notice that $Q_{\ul{a}}(\ul{\alpha})Q_{\ul{1}-\ul{a}}(\ul{\alpha})=Q_{\ul{0}}(\ul{\alpha})Q_{\ul{1}}(\ul{\alpha})$ for any $\ul{a}\in\Digit'$.
By Lemma~\ref{lemma:diagonalization}, we have
\[
\begin{array}{r@{\ }c@{\ }l}
\det A(\ul{\alpha})&=&
\displaystyle
\prod_{\ul{a}\in\mathcal{A}}(-1)^m\cdot 2^r\ \sign\pcos{\ul{\alpha}}
=
\prod_{\ul{a}\in\Digit'}
2^{2r}\,(\sqrt{-1})^{r}Q_{\ul{a}}(\ul{\alpha})Q_{\ul{1}-\ul{a}}(\ul{\alpha})\\
&=&
\displaystyle
\Bigl(2^{2r}(\sqrt{-1})^r\prod_{j=1}^r\cos\frac{\pi \alpha_j}{2}\,\sin\frac{\pi \alpha_j}{2}\Bigr)^{2^{r-1}}
=
\Bigl(\prod_{j=1}^r2\sqrt{-1}\sin\pi\alpha_j\Bigr)^{2^{r-1}}.
\end{array}
\]
On the other hand,
an explicit formula of $\det A(\ul{\alpha})$ is not given for general $\Omega$.
Let us consider for the case of the Vinberg cone $\Omega$.
Namely, $\Omega$ is a homogeneous cone of rank $3$ with
$n_{21}=n_{31}=1$ and $n_{32}=0$ so that the condition~\eqref{assumption} fails.
In this case, 
the determinant of $A(\ul{\alpha})$ is calculated as
\[\det A(\ul{\alpha})=
2^{12}\bigl(\sin\pi\alpha_1\bigr)^4\bigl(\sin\pi\alpha_2\bigr)^4\bigl(\sin\pi\alpha_3\bigr)^4=
\Bigl(\prod_{j=1}^32\sqrt{-1}\sin\pi\alpha_j\Bigr)^4,\]
and therefore, for any homogeneous cone of rank $r\ge 2$, we shall conjecture
\[
\det A(\ul{\alpha})=
\Bigl(\prod_{j=1}^r2\sqrt{-1}\sin\pi \alpha_j\Bigr)^{2^{r-1}}.
\]
\end{remark}

\section{Zeta distributions}
\label{sect:zeta distribution}

In this section,
we investigate a relationship between ${}^{t\!}J\vlzf{f}{\s}$ in the main theorem and zeta distributions defined in~\eqref{def:zd} below.
Let $\mathrm{sgn}$ be the sign function of $\R^\times$, that is,
$\mathrm{sgn}(x):=x/|x|$, and we set $\mathrm{sgn}(0):=0$.
Let $\omega^{s,a}$ $(s\in\C,\ a=0,1)$ be the quasi-character of $\R^{\times}$ defined by
$\omega^{s,a}(x):=\mathrm{sgn}(x)^a\,|x|^s$.
Using this,
we introduce zeta distributions $\localZ{f}$ by
\begin{equation}
\label{def:zd}
\localZ{f}:=\int_V\prod_{j=1}^r\omega^{s_j,b_i}\bigl(\Delta_j(x)\bigr)\,f(x)\,d\mu(x)
\quad
(\ul{b}\in\Digit,\ f\in\rapid,\ \ul{s}\in\C^r).
\end{equation}
We note that, 
for symmetric cones viewed as homogeneous spaces of reductive groups,
zeta distributions are studied in \cite{BCK2018}.
Let us describe $\localZ{f}$ by using $\lzf{f}{\s}$.
Set $\ce:=\mathrm{diag}(\varepsilon_1I_{n_1},\dots,\varepsilon_rI_{n_r})\in V$ (cf.\ \eqref{eqV})
and let $\ul{e}_j\in\R^r$ be the row unit vector having one on the $j$-th position and zeros elsewhere.
Recalling the property~\eqref{eq:multiplier matrix} of the multiplier matrix $\sigma=(\sigma_{jk})_{1\le j,k\le r}$,
we have
\[
\Delta_j(\ce)=\varepsilon_1^{\sigma_{j1}}\cdots\varepsilon_r^{\sigma_{jr}}=\kvsup{\ul{e}_j\sigma}.
\]
Since $\Oe$ is the orbit of $H$ through $\ce\in V$
and since $\Delta_j$ is a relatively invariant function,
we have for $x\in\Oe$
\[
\omega^{s_j,b_j}\bigl(\Delta_j(x)\bigr)
=
\bigl(\varepsilon_1^{\sigma_{j1}}\cdots\varepsilon_r^{\sigma_{jr}}\bigr)^{b_j}
|\Delta_j(x)|^{s_j}
=\kvsup{b_j\,\ul{e}_j\sigma}|\Delta_j(x)|^{s_j},
\]
and thus
\[
\prod_{j=1}^r\omega^{s_j,b_j}\bigl(\Delta_j(x)\bigr)
=
\kvsup{\ul{b}\sigma}\,|\Delta_1(x)|^{s_1}\cdots|\Delta_r(x)|^{s_r}.
\]
By~\eqref{eq:lzf}, we obtain
\begin{equation}
\label{eq:linearcomb}
\localZ{f}=\sum_{\ue\in\Ir}
\kvsup{\ul{b}\sigma}
\lzf{f}{\ul{s}},
\end{equation}
whence each $\localZ{f}$ is one of entries of ${}^{t\!}J\vlzf{f}{\s}$.
This formula also tells us that
analytic properties of $\localZ{f}$ are the same as $\lzf{f}{\s}$.
Since the map
$\Digit\ni\ul{a}\mapsto\ul{a}\sigma\ (\mathrm{mod}\ 2)\in\Digit$
is a bijection because $\sigma$ is a unimodular matrix,
the correspondence of $\localZ{f}$ and entries of ${}^{t\!}J\vlzf{f}{\s}$ is one-to-one.

In order to state the main theorem by using zeta distributions,
we need to introduce another order in $\Digit$.
Let us denote by $<$ the fixed order in $\Digit$.
According to \eqref{eq:linearcomb},
we introduce a new order $\prec_\sigma$ in $\Digit$ depending on $\sigma$ by
\[
\ul{b}\prec_\sigma\ul{b}'\quad\Leftrightarrow\quad
\ul{b}\sigma<\ul{b}'\sigma\quad(\ul{b},\ul{b}'\in\Digit).
\]
Let $\Digit_\sigma$ denote the set $\Digit$ with the order $\prec_\sigma$.
Then,
we have
\[
\bs{Z}(f;\,\ul{s}):=
\bigl(\localZ{f}\bigr)_{\ul{b}\in\Digit_\sigma}
=
{}^{t\!}J\vlzf{f}{\ul{s}}.
\]
On the other hand,
to the dual cone $\Omega^*$,
we can associate zeta distributions $\dlocalZ{f}$ $(\ul{c}\in\Digit)$, 
and similarly to $\localZ{f}$,
they satisfy
\[
\dlocalZ{f}=\sum_{\ud\in\Ir}\kappa_{\ud}(\ul{c}\sigma_*)\dlzf{f}{\ul{s}}
\quad(\ul{c}\in\Digit,\ f\in\rapid,\ \ul{s}\in\C^r).
\]
Therefore, we need to consider $\Digit_{\sigma_*}$ so that
\[
\bs{Z}^*(f;\,\ul{s}):=\bigl(\dlocalZ{f}\bigr)_{\ul{c}\in\Digit_{\sigma_*}}
=
{}^{t\!}J\vdlzf{f}{\ul{s}}.
\]

\begin{corollary}
Assume the condition~\eqref{assumption}.
Putting
\[
\widetilde{\bs{Z}}(f;\,\ul{s}):=
\Lambda\bigl(\s\sigma-2^{-1}\ul{p}\bigr)\bs{Z}(f;\,\ul{s}),\quad
\widetilde{\bs{Z}}{}^*(f;\,\s):=
\Lambda\bigl(\s\sigma_*-2^{-1}\ul{q}\bigr)\bs{Z}^*(f;\,\ul{s})
\]
for $\ul{s}\in\C^r$ and $f\in\rapid$, one has
\[
\widetilde{\bs{Z}}(\Fourier;\,\s)
=
\efactor\,\widetilde{\bs{Z}}{}^*(f;\,\tau(\s)),\quad
\efactor=(-1)^m\mathrm{diag}\Bigl(\sign\Bigr)_{\ul{a}\in\Digit}.
\]
\end{corollary}

We end this paper by giving a remark on orders of zeta distributions.
In this paper,
we have fixed orders of the basic relative invariants $\Delta_i$, $\Delta^*_j$ of $\Omega$ and $\Omega^*$, respectively, as in the previous paper~\cite{N2018}.
In the case that $\Omega$ is symmetric,
we have the canonical order among $\Delta^*_j$ as in the book~\cite[Chapter VII]{FK94}, which is the opposite order of ours,
and if we use the canonical one,
then we see that there is no necessity to specify an order in $\mathcal{A}$.
For general homogeneous cones $\Omega$,
let us consider a natural generalization of the canonical orders used in symmetric cases, that is,
we take the opposite order among $\Delta^*_j$.
In this case,
the rearranged order of $\Digit_{\sigma_*}$ is equal to $\Digit_{\sigma}$
if and only if we have $A\sigma_*A^{-1}=\sigma$ where $A$ is the anti-diagonal matrix $A=\bigl(\delta_{i,r-j+1}\bigr)_{i,j}$ of size $r$.
Here, the multiplier matrix $\sigma_*$ is determined by using the original order used in~\cite{N2018}.
Note that $\delta_{ij}$ is the Kronecker delta.
However, as in Example~\ref{ex} below,
there exist homogeneous cones such that $A\sigma_*A^{-1}\ne \sigma$,
and therefore,
for general homogeneous cones,
we do not have canonical orders such as 
those of symmetric cones.

\begin{example}
\label{ex}
Let $V$ be a vector space defined by
\[
V:=\set{x=\pmat{
x_{11}I_4&0&\bs{x}_{21}&0\\
0&x_{11}I_4&0&\bs{x}_{31}\\
{}^{t\!}\bs{x}_{21}&0&x_{22}&0\\
0&{}^{t\!}\bs{x}_{31}&0&x_{33}
}}{\begin{array}{l}x_{11},x_{22},x_{33}\in\R,\\
\bs{x}_{21},\bs{x}_{31}\in\R^4\end{array}}\subset \mathrm{Sym}(10,\R).
\]
Then, the set
\[
\Omega:=\set{x\in V}{x\text{ is positive definite}}
\]
is a homogeneous cone of rank $3$
whose structure constants are $n_{21}=n_{31}=4$ and $n_{32}=0$.
The basic relative invariants $\Delta_1,\Delta_2,\Delta_3$ of $\Omega$ are 
\[
\Delta_1(x)=x_{11},\quad
\Delta_2(x)=x_{11}x_{22}-|\bs{x}_{21}|^2,\quad
\Delta_3(x)=x_{11}x_{33}-|\bs{x}_{31}|^2,
\]
and those $\Delta^*_1,\Delta^*_2,\Delta^*_3$ of $\Omega^*$ are
\[
\Delta^*_1(y)=y_{11}y_{22}y_{33}-y_{22}|\bs{y}_{31}|^2-y_{33}|\bs{y}_{21}|^2,\quad
\Delta^*_2(y)=y_{22},\quad
\Delta^*_3(y)=y_{33}.
\]
The multiplier matrices $\sigma$ and $\sigma_*$ are therefore given as
\[\sigma=\pmat{1&0&0\\1&1&0\\1&0&1},\quad
\sigma_*=\pmat{1&1&1\\0&1&0\\0&0&1},
\]
whence we see that
\[
A\sigma_*A^{-1}=\pmat{1&0&0\\0&1&0\\1&1&1}\ne
\sigma.
\]
\end{example}

\section*{Acknowledgments}

This work was supported by Grant-in-Aid for JSPS fellows (2018J00379).
The present author is grateful to Professor Hiroyuki Ochiai for insightful comments for this work.
He also would like to express his sincere gratitude to Professor Kohji Matsumoto for the encouragement and the advice in writing this paper.

\end{document}